\documentclass[leqno,final]{siamltex704}
\usepackage{amssymb,amsmath,amsfonts}
\usepackage{graphicx}
\graphicspath{{./pics/}}
\usepackage{subfig}
\usepackage{float}
\usepackage{verbatim}
\usepackage[colorlinks,linktocpage,linkcolor=blue]{hyperref}
\usepackage{algorithm,algorithmic}
\usepackage{enumerate}

\usepackage{framed,color,xcolor}

\definecolor{shadecolor}{gray}{0.80}

\newtheorem{thm}{Theorem}[section]
\newtheorem{rmk}{Remark}[section]
\newtheorem{lem}{Lemma}[section]
\newtheorem{ass}{Assumption}[section]
\newtheorem{prop}{Proposition}[section]

\newtheorem{exam}{Example}

\newcommand{\bs}{\boldsymbol}

\begin{document}
\title{Multi-Parameter Tikhonov Regularization\\ -- An Augmented Approach}
\author{Kazufumi Ito\thanks{Center for Research in Scientific Computation \&
Department of Mathematics, North Carolina State University, Raleigh, North Carolina
27695, USA. (kito@math.ncsu.edu)} \and Bangti Jin\thanks{Department of Mathematics, Texas A\&M University,
College Station, Texas 77843-3368, USA. ({\texttt btjin@math.tamu.edu})}\and Tomoya
Takeuchi\thanks{Collaborative Research Center for Innovative Mathematical Modelling,
Institute of Industrial Science, The University of Tokyo
4-6-1-Cw601 Komaba, Meguro-ku, Tokyo 153-8505, Japan. ({\texttt takeuchi@sat.t.u-tokyo.ac.jp})}}

\maketitle
\begin{abstract}
We study multi-parameter regularization (multiple penalties) for solving linear inverse
problems to promote simultaneously distinct features of the sought-for objects. We
revisit a balancing principle for choosing regularization parameters from the viewpoint
of augmented Tikhonov regularization, and derive a new parameter
choice strategy called the \textit{balanced discrepancy principle}.
A priori and a posteriori error
estimates are provided to theoretically justify the principles, and numerical algorithms
for efficiently implementing the principles are also provided. Numerical results on
denoising are presented to illustrate the feasibility of the balanced discrepancy principle.\\
\textbf{Keywords}:
multi-parameter regularization, augmented Tikhonov regularization,
balanced discrepancy principle
\end{abstract}

\begin{AMS}
65J20, 65J22, 49N45
\end{AMS}

\section{Introduction}
We investigate a regularization technique for robustly solving linear inverse problems modeled
by
\begin{equation}\label{eqn:lininv}
Ku^\dagger=g^\dagger,
\end{equation}
where $g^\dagger$ is the (inaccessible) exact data and $u^\dagger \in X$ represents the unknown
exact solution, and $K:X\rightarrow Y$ is a bounded linear operator. Here the spaces $X$ and
$Y$ are general Banach spaces, and the operator $K$ can be an embedding operator (image denoising), a convolution
operator (deblurring, scattering) and the Radon transform (computed tomography).
The objective is to find an approximation $u$ to the solution $u^\dagger$ from noisy measurement $g^\delta\in Y $ of
the exact data $g^\dagger$. The accuracy of the noisy data $g^\delta$ is measured by the standard
$\mathrm{L}^2$ fidelity functional
$\phi(u^\dagger,g^\delta)=\frac{1}{2}
\|Ku^\dagger-g^\delta\|^2=\frac{1}{2}\delta^2$ with the noise level $\delta$.

As is typical for many inverse problems, problem \eqref{eqn:lininv} suffers from ill-posedness
or instability. This poses significant challenges to their
accurate yet stable numerical solution in the presence of data noise, which is often the
case in practical applications. Often, regularization is applied to find a stable approximate solution.
One of the most widely used approaches is known as Tikhonov regularization. It seeks to minimize the following functional
\begin{equation}\label{eqn:multi}
J_{\bs\eta}(u)=\phi(u,g^\delta)+\bs{\eta}\cdot\bs{\psi}(u),
\end{equation}
over a closed convex feasible solution set $\mathcal{C}$. The solution to
the minimization problem, denoted by $u_{\mathcal{\bs\eta}}^\delta$ ($u_{\bs\eta}$ in case of the exact
data $g^\dagger$), serves as an approximation to the exact solution $u^\dagger$. Here the
(nonnegative) vector-valued penalty functional $\bs\psi$ encodes the a
priori knowledge, and $\bs{\eta} \cdot \bs\psi(u)$ denotes the dot product
between the regularization parameter vector
$\bs{\eta}=(\eta_1,\eta_2)^\mathrm{t}\in\mathbb{R}^2_+$ and the
penalty $\bs\psi(u) = (\psi_1(u),\psi_2(u))^\mathrm{t}$. The penalty
$\bs\psi$ is selected to promote desirable features of the sought-for solution, e.g.,
edge, sparsity and texture; and often the optimization problem \eqref{eqn:multi}
is nonsmooth. The (vector) parameter $\bs\eta$
compromises the fidelity $\phi$ with the penalty $\bs\psi$, and its appropriate
choice plays a crucial role in obtaining stable yet accurate solutions.
Therefore, an automated selection rule and efficient algorithms for determining $\bs\eta$
are essential.

One distinct feature of the model \eqref{eqn:multi} is that it includes multiple 
penalties (hence termed as multi-parameter regularization). This is motivated by 
the following empirical observations. In practice,
many objects exhibit distinct multiple features/structures.
However, one single penalty generally favors one feature over others, and thus
unsuitable for promoting multiple distinct features. For example, total variation ($\mathrm{TV}$) is
well suited to reconstructing piecewise constant structures, however, it results in significant
staircases in gray regions. One may improve $\mathrm{TV}$-reconstruction by introducing an
additional penalty, say $\mathrm{L}^1$ norm of $\Delta u$ where $\Delta$ is the
Laplacian operator. Hence, a reliable recovery of several distinct features naturally 
calls for multiple penalties, and it is not surprising that the idea of 
multi-parameter regularization has been pursued earlier. For instance, in 
\cite{ItoKunisch:2000} the authors proposed a model to preserve both flat and gray 
regions in natural images by combining $\mathrm{TV}$ with Sobolev smooth penalty. 
We refer interested readers to \cite{Stephanakis:1997,LuShenXu:2007} (imaging),
\cite{ZouHastie:2005} (microarray data analysis), \cite{XuFukudaLiu:2006} (geodesy)
and \cite{LuPereverzev:2009} (machine learning) for other interesting applications.

However, a general theory of multi-parameter regularization remains under 
development \cite{BelgeKilmerMiller:2002,ChenLuXuYang:2008, LuPereverzev:2009,
ItoJinTakeuchi:2011}. In \cite{BelgeKilmerMiller:2002} the $L$-hypersurface
was suggested for determining regularization parameters for
finite-dimensional linear systems, but without any theoretical
justification. In \cite{ChenLuXuYang:2008}, a multi-resolution analysis for ill-posed
linear operator equations was analyzed, and some convergence results were established. Lu
et al. \cite{LuPereverzev:2009} discussed the discrepancy principle for Hilbert space
scales, and derived some error estimates. However, the parameter selection is vastly nonunique
due to lack of constraints and thus not directly applicable in practice, for which later a
quasi-optimality criterion was suggested \cite{LuPereverzevShaoTautenhahn:2010}. Recently, the
authors \cite{ItoJinTakeuchi:2011} investigated the discrepancy principle and a
balancing principle for general convex variational
models. However, the nonuniqueness of the discrepancy
principle remains unresolved, and further, there is still no theory for the balancing principle for
multi-parameter regularization.

The present work extends our earlier work \cite{ItoJinTakeuchi:2011}, and
includes the following essential contributions.
We first revisit the balancing principle in \cite{ItoJinTakeuchi:2011} from the
viewpoint of augmented Tikhonov regularization \cite{JinZou:2009}, and established 
the equivalence. Then we derive a novel hybrid principle, the balanced discrepancy principle, by
incorporating constraints into the augmented approach, which partially resolves the
nonuniqueness issue. Further, a priori and a posterior error estimate are derived for both
principles. The estimate in Theorem \ref{thm:errest} was stated in \cite{ItoJinTakeuchi:2011} 
without a proof. Finally, we develop efficient algorithms for
implementing these principles, and briefly discuss their properties.

The rest of the paper is organized as follows. In \S \ref{sec:atikh}, we derive
the balancing principle and the new hybrid principle, and develop relevant error estimates. 
In \S \ref{sec:alg} we discuss efficient implementations of the two principles.
Finally, we provide some numerical results to illustrate the hybrid principle in \S \ref{sec:exp}.

\section{An augmented approach}\label{sec:atikh}
The augmented Tikhonov (a-Tikhonov) regularization is one principled framework
for choosing regularization parameters \cite{JinZou:2009}.
Here we describe the augmented approach for multi-parameter models, and
derive the balancing principle and a novel balanced discrepancy principle.

\subsection{Derivation of the principles}\label{ssec:atikh}
\subsubsection{Balancing principle}\label{sssect:atikh}
First we sketch the augmented approach. For the multi-parameter model
\eqref{eqn:multi}, it can be derived analogously from hierarchical Bayesian inference as
in \cite{JinZou:2009}, and the resulting augmented functional
$J(u,\tau,\bs\lambda)$ reads
\begin{equation*}
    J(u,\tau,\bs\lambda)=\tau\phi(u,g^\delta)+ \bs{\lambda}\cdot\bs{\psi}(u)+
    \bs{e}\cdot(\beta\bs\lambda-\alpha\ln\bs\lambda) + \beta_0\tau-\alpha_0 \ln\tau,
\end{equation*}
where the vector $\bs{e}$ is given by $\bs{e}=(1,1)^\mathrm{t}$. The functional 
$J(u,\tau,\bs\lambda)$ maximizes the posteriori probability density function
\begin{equation*}
p(u,\tau,\bs\lambda|g^\delta)\propto\, p(g^\delta|u,\tau,\bs\lambda)\,
p(u,\tau,\bs\lambda).
\end{equation*}
The functional $J(u,\tau,\bs\lambda)$ is derived under the assumption that the scalars
$\lambda_i$ and $\tau$ have Gamma distributions with known parameter pairs.
The parameter pairs $(\alpha,\beta)$ and $(\alpha_0,\beta_0)$ are related to the
shape parameters in the statistical priors on the prior precision $\lambda_i$ and noise
precision $\tau$, respectively. The special case $\beta_0=\beta=0$
is known as noninformative prior and customarily adopted in practice. Hence we focus our derivation
on this case. Upon letting $\eta_i=\frac{\lambda_i}{\tau}$, the
necessary optimality condition of any minimizer $(u_{\bs\eta}^\delta,\lambda_i,\tau)$ to
the a-Tikhonov functional $J(u,\tau,\{\lambda_i\})$ is given by
\begin{equation}\label{eqn:atikhopt_cc}
  \left\{\begin{aligned}
     u^\delta_{\bs\eta}&=\arg\min_{u\in\mathcal{C}}\;\left\{\phi(u,g^\delta)+\bs{\eta}\cdot\bs{\psi}(u)\right\},\\
        \lambda_i&=\dfrac{\alpha}{\psi_i(u^\delta_{\bs\eta})},\quad i=1,2,\\
        \tau&=\dfrac{\alpha_0}{\phi(u_{\bs\eta}^\delta,g^\delta)}.
  \end{aligned}\right.
\end{equation}
Now by rewriting the system with $\gamma=\frac{\alpha_0}{\alpha}$, we arrive at the following system for
$(u^\delta_{\bs\eta},{\bs\eta})$
\begin{equation}\label{eqn:atikhopt}
  \left\{\begin{aligned}
    &u^\delta_{\bs\eta}=\arg\min_{u\in\mathcal{C}}\;\left\{\phi(u,g^\delta)+{\bs\eta}\cdot{\bs\psi}(u)\right\},\\
    &\eta_i=\frac{1}{\gamma}\frac{\phi(u^\delta_{\bs\eta},g^\delta)
      }{\psi_i(u^\delta_{\bs\eta})},\ i=1,2.
  \end{aligned}\right.
\end{equation}
The optimality system \eqref{eqn:atikhopt} reveals the mechanism of the augmented
approach: it selects an optimal regularization parameter $\bs\eta$ in the model
\eqref{eqn:multi} by balancing the penalty $\bs\psi$ with the
fidelity $\phi$, from which the term balancing principle follows. We note the term
balancing principle here should not be confused with Lepskii's principle, which is also sometimes
called a balancing principle \cite{Mathe:2006}. The Lepskii's principle does require
a knowledge of noise level.

Next we characterize \eqref{eqn:atikhopt} using
the value function $F(\bs\eta)$ \cite{Ito:2010a} defined by
\begin{equation*}
F(\bs\eta)=\inf_{u\in\mathcal{C}} J_{\bs\eta}(u).
\end{equation*}
The function $F(\bs\eta)$ is continuous, and it is almost everywhere
differentiable, cf. Lemma \ref{lem:diff}. We denote by
$F_{\eta_i}$ the partial derivative of $F(\bs\eta)$ with respect to $\eta_i$. The proof
is analogous to \cite{Ito:2010a}, and hence omitted.

\begin{lem}\label{lem:diff}
The function $F(\bs\eta)$ is monotone and concave, and hence almost everywhere
differentiable. Further, if it is differentiable, then there holds
$F_{\eta_i}(\bs\eta)=\psi_i(u_{\bs\eta}^\delta).$
\end{lem}

Next we provide an alternative characterization
of \eqref{eqn:atikhopt}. First we define the function $\Phi_\gamma(\bs\eta)$ by
\begin{equation}\label{eqn:pPhi}
  \Phi_\gamma({\bs \eta})=\frac{F({\bs \eta})^{\gamma+2}}{\eta_1\eta_2}.
\end{equation}
The necessary optimality condition for
$\Phi_{\gamma}({\bs \eta})$, provided that $F({\bs \eta})$ is differentiable,
reads
\begin{equation*}
        \dfrac{\partial \Phi_\gamma}{\partial \eta_i}=\dfrac{F({\bs \eta})^{\gamma+1
        }}{\eta_1\eta_2}\dfrac{\left(-
        F({\bs \eta})+(2+\gamma)\eta_i F_{\eta_i}
        ({\bs \eta})\right)}{\eta_i}=0, \quad i=1,2
\end{equation*}
which, upon noting Lemma \ref{lem:diff}, is equivalent to
\begin{equation*}
  \left\{\begin{array}{l}
     -\phi(u_{\bs\eta}^\delta,g^\delta)  + (1+\gamma)\eta_1\psi_1(u_{\bs\eta}^\delta) - \eta_2 \psi_2(u_{\bs\eta}^\delta)=0,\\\\
     -\phi(u_{\bs\eta}^\delta,g^\delta)  - \eta_1 \psi_1(u_{\bs\eta}^\delta) + (1+\gamma)\eta_2\psi_2(u_{\bs\eta}^\delta)=0.
   \end{array}\right.
\end{equation*}
Solving the system with respect to $\eta_i$ yields $
\eta_i = \frac{1}{\gamma}\frac{\phi(u_{\bs\eta}^\delta,g^\delta)}{\psi_i(u_{\bs\eta}^\delta)}$.
Hence, the optimality system of the function $\Phi_\gamma$ coincides with that of the
functional $J(u,\tau,\bs\lambda)$. In summary, we have shown our first main result.
\begin{prop}
Let the value function $F(\bs\eta)$ be differentiable. Then all critical points of
the function $\Phi_\gamma$ are solutions to system \eqref{eqn:atikhopt}.
\end{prop}

\begin{rmk}
Two remarks on the function $\Phi_\gamma$ are in order. First, it
is very flexible in that the free-parameter $\gamma$ may be calibrated to achieve specific desirable
properties. Second, by the concavity in Lemma \ref{lem:diff}, $F(\bs\eta)$ is continuous
and thus the problem of minimizing $\Phi_\gamma$ over any bounded and closed region in
$\mathbb{R}^2_+$ is well defined. These observations remain valid for a general fidelity.
\end{rmk}

\subsubsection{Balanced discrepancy principle}
To solve stably and accurately problem \eqref{eqn:lininv}, one should use all
prior information, e.g., the noise level
$\phi(u^\dagger,g^\delta)= c:=\frac{1}{2}c_m^2\delta^2$ for some $c_m\geq1$, and other relevant knowledge, whenever it is
available. This can be realized by incorporating constraints into the augmented approach,
and then deriving the corresponding optimal system. For instance, for the
constraint $\phi(u,g^\delta)\leq c$, the Lagrangian approach gives
the following a-Tikhonov functional
\begin{equation*}
  \begin{aligned}
    J(u,\tau,\bs\lambda,\mu)=\tau\phi(u,g^\delta)+ &\bs{\lambda}\cdot\bs{\psi}(u)
      -\alpha\bs{e} \cdot\ln\bs\lambda
     -\alpha_0 \ln\tau + \tau\langle \phi(u,g^\delta)-c,\mu\rangle,
  \end{aligned}
\end{equation*}
where the unknown scalar $\mu \geq 0$ is the Lagrange multiplier for the inequality
constraint $\phi(u,g^\delta)\leq c$. Its optimality system reads
\begin{equation*}
  \left\{
  \begin{aligned}
   u^\delta_{\bs\eta}&=\arg\min_{u}\;\left\{\phi(u,g^\delta)+\bs{\eta}\cdot\bs{\psi}(u)
   +\langle\phi(u,g^\delta)-c,\mu\rangle\right\},\\
        \lambda_i&=\dfrac{\alpha}{\psi_i(u^\delta_{\bs\eta})},\quad i=1,2,\\
        \tau&=\dfrac{\alpha_0}{(1+\mu)\phi(u_{\bs\eta}^\delta,g^\delta)},\\
        c&\geq\phi(u_{\bs\eta}^\delta,g^\delta), \ \mu\geq 0.
  \end{aligned}\right.
\end{equation*}
Hence the constraint $\phi(x,g^\delta)\leq c$ and the balancing principle are both
fulfilled:
\begin{equation}\label{eqn:rbsys}
\gamma\eta_i\psi_i(u_{\bs\eta}^\delta)=(1+\mu)\phi(u_{\bs\eta}^\delta,g^\delta),\quad i=1,2.
\end{equation}

In the case of one single penalty, identity \eqref{eqn:rbsys} does not provide any
additional constraint since the multiplier $\mu$ is also unknown. We observe that
the active constraint, i.e.,
$\|Ku_{\bs\eta}^\delta-g^\delta\|=c_m\delta$, is exactly the discrepancy
principle \cite{EnglHankeNeubauer:1996}. The constraint is active under certain
conditions \cite{IvanovVasinTanana:2002}. Nonetheless, in case of multiple penalties, the discrepancy
principle alone cannot uniquely determine $\bs\eta$. Hence
we include also system \eqref{eqn:rbsys}, which might help resolve the
nonuniqueness issue. Upon simplification, this yields a new hybrid principle
\begin{equation} \label{aT}
  \left\{\begin{aligned}
     &\phi(u_{\bs\eta}^\delta,g^\delta)=\tfrac{1}{2}c_m^2\delta^2,\\
     &\eta_1\psi_1(u_{\bs\eta}^\delta)=\eta_2\psi_2(u_{\bs\eta}^\delta).
  \end{aligned}\right.
\end{equation}

The principle can be interpreted as the augmented approach with the
constraint $\{u:\|Ku-g^\delta\|=c_m\delta\}$, $c_m\geq 1$. Hence
it integrates the classical discrepancy principle
$\|Ku_{\bs\eta}^\delta-g^\delta\| = c_m\delta$ with the balancing
principle, and we shall name the new rule \eqref{aT} \textit{ balanced discrepancy
principle}. One noteworthy feature of \eqref{aT} is that it does not involve
the free parameter $\gamma$.

\subsection{Error estimates}\label{ssec:err}

Now we derive error estimates for \eqref{eqn:pPhi} and \eqref{aT}, capitalizing
on \cite{EnglHankeNeubauer:1996,BurgerOsher:2004,Hofmann:2007}. We discuss the following
three scenarios separately: hybrid principle \eqref{aT}, purely balancing
principle \eqref{eqn:pPhi} in Hilbert and Banach spaces. These theoretical results
partially justify their practical usages.

\subsubsection{Balanced discrepancy principle}

In this part, we discuss the consistency and an a priori error estimate for the hybrid
principle \eqref{aT}. To this end, we make the following assumption.
\begin{ass}\label{ass:phipsi}
There exists a $\tau$-topology such that for any $\bs\eta>0$, the functional
$J_{\bs\eta}(u)$ is coercive and its level set $\{u\in\mathcal{C}: J_{\bs\eta}(u)\leq
c\}$ for any $c>0$ is compact in $\tau$-topology, and the functionals $\phi$ and $\psi_i$
are $\tau$ lower semi-continuous.
\end{ass}
\begin{rmk}
The $\tau$-topology is naturally induced by the penalty functional $\bs\psi$, and it is not
arbitrarily in order to ensure the lower semicontinuity. 
\end{rmk}

Now we can state a consistency result. The line of proof is standard \cite{ItoJinTakeuchi:2011}, and thus omitted.
\begin{thm}\label{thm:condp}
Let Assumption \ref{ass:phipsi} be fulfilled, and
$t(\bs\eta)=\frac{\eta_1(\delta)}{\eta_1 (\delta)+\eta_2(\delta)}$. Let the sequence
$\{\bs\eta(\delta)\}_\delta$ be selected by \eqref{aT}. If a subsequence of
$\{\bs\eta(\delta)\}_\delta$ converges and $\widetilde{t}:=\lim_{\delta\rightarrow0}
t(\delta)\in (0,1)$, then the subsequence $\{u_{\bs\eta(\delta)}^\delta\}_\delta$
contains a subsequence $\tau$-converging to a
$[\widetilde{t},1-\widetilde{t}]^\mathrm{t}\cdot\bs\psi$-minimizing solution of
$Ku=g^\dagger$ and
\begin{equation*}
\lim_{\delta\rightarrow0}[t(\delta),1-t(\delta)]^\mathrm{t}\cdot\bs\psi(u_{\bs\eta}^\delta)
=[\widetilde{t},1-\widetilde{t}]^\mathrm{t}\cdot\bs\psi(u^\dagger).
\end{equation*}
\end{thm}

\begin{rmk}
The condition $\tilde{t}\in(0,1)$ in
Theorem \ref{thm:condp} amounts to the uniform boundedness of $\psi_i(
u_{\bs\eta}^\delta)$.
\end{rmk}

Next we have the following convergence rate, i.e., the distance between 
the approximation $u_{\bs\eta}^\delta$ and
the true solution $u^\dagger$ (in Bregman distance
\cite{BurgerOsher:2004}) in terms of the noise level $\delta$. We denote the 
subdifferential of a convex functional $\psi(u)$ at $u^\dagger$ by
$\partial\psi(u^\dagger)$, i.e.,
\begin{equation*}
   \partial\psi(u^\dagger)=\{\xi\in X^\ast: \psi(u)\geq\psi(u^\dagger)+\langle \xi,u-u^\dagger\rangle,\, \forall u\in X\},
\end{equation*}
and the Bregman distance $d_\xi(u,u^\dagger)$
for any $\xi\in\partial\psi(u^\dagger)$ is defined as
\begin{equation*}
d_\xi(u,u^\dagger):=\psi(u)-\psi(u^\dagger)-\langle
\xi,u-u^\dagger\rangle.
\end{equation*}

Now we can state a convergence rates result.
\begin{thm}\label{thm:disc}
Let the exact solution $u^\dagger$ satisfy the source condition: for any $t\in[0,1]$,
there exists a $w_t\in Y$ such that
$K^\ast w_t=\xi_t\in\partial\left([t,1-t]^\mathrm{t}\cdot\bs\psi(u^\dagger)\right).
$
Then for any $\bs\eta^\ast$ determined by the principle \eqref{aT} and with 
$t^\ast=t(\bs\eta^*)=\tfrac{\eta_1^\ast(\delta)}{\eta_1^\ast (\delta) +\eta_2^\ast 
(\delta)}\in[0,1]$, the following estimate holds
\begin{equation*}
d_{\xi_{t^*}}(u_{\bs\eta^\ast}^\delta,u^\dagger)\leq (1+c_m)\|w_{t^\ast}\|\delta.
\end{equation*}
\end{thm}
\begin{proof}
The line of proof is again well known, but we include a sketch for completeness.
In view of the minimizing property of the approximation $u_{\bs\eta^\ast}^\delta$ and the
constraint $\|Ku_{\bs\eta^\ast}^\delta-g^\delta\|=c_m\delta$, we have
$
[t^*,1-t^\ast]^\mathrm{t}\cdot\bs\psi(u_{\bs\eta^\ast}^\delta)\leq [t^*,1-t^\ast]^\mathrm{t}\cdot\bs\psi(u^\dagger).
$
The source condition implies that there exists a
$\xi_{t^\ast}\in\partial\left([t^\ast,1-t^\ast]^\mathrm{t}\cdot\bs\psi(u^\dagger)\right)$
and $w_{t^\ast}\in Y$ such that $\xi_{t^\ast}=K^\ast w_{t^\ast}$. From this
and the Cauchy-Schwarz inequality, we deduce
\begin{equation*}
\begin{aligned}
d_{\xi_{t^*}}(u_{\bs\eta^\ast}^\delta,u^\dagger)&=[t^\ast,1-t^\ast]^\mathrm{t}\cdot\bs\psi(u_{\bs\eta^\ast}^\delta)
-[t^\ast,1-t^\ast]^\mathrm{t}\cdot\bs\psi(u^\dagger)-\langle\xi_{t^*},u_{\bs\eta^\ast}^\delta-u^\dagger\rangle\\
&\leq-\langle\xi_{t^*},u_{\bs\eta^\ast}^\delta-u^\dagger\rangle
=-\langle K^\ast w_{t^*},u_{\bs\eta^\ast}^\delta-u^\dagger\rangle\\
&=-\langle w_{t^*},K(u_{\bs\eta^\ast}^\delta-u^\dagger)\rangle
\leq\|w_{t^*}\|\|K(u_{\bs\eta^\ast}^\delta-u^\dagger)\|\\
&\leq\|w_{t^*}\|\left(\|Ku_{\bs\eta^\ast}^\delta-g^\delta\|+\|g^\delta-Ku^\dagger\|\right)\leq
(1+c_m)\|w_{t^*}\|\delta.
\end{aligned}
\end{equation*}
This shows the desired estimate.
\end{proof}

\begin{rmk}
In Theorem \ref{thm:disc}, the order of convergence relies solely on the
constraint  $\|Ku_{\bs\eta}^\delta-g^\delta\|=c_m\delta$,
while the weight $t^\ast$ in the estimate is determined by the balancing
principle. Hence the reduced system \eqref{eqn:rbsys} does help
resolve the vast nonuniqueness issue in the discrepancy principle.
\end{rmk}

\subsubsection{Balancing principle in Hilbert spaces}
We derive a posteriori estimates for the balancing principle $\Phi_\gamma$
\eqref{eqn:pPhi}, i.e., the distance between the approximation $u_{\bs\eta^\ast}^\delta$ and the
exact solution $u^\dagger$ in terms of the noise level $\delta=\|g^\delta-g^\dagger\|$
and the realized residual $\delta_*=\|Ku_{\bs\eta^\ast}^\delta-g^\delta\|$ etc. We first
treat quadratic regularizations $\psi_i(u)=\frac{1}{2}\|L_iu\|^2$ with linear
operators $L_i$ fulfilling $\ker(L_i)\cap\ker(K)=\{0\}$, $i=1,2$, and each induces a
semi-norm. One typical choice is that $\psi_1$ and $\psi_2$ impose the $\mathrm{L}^2$-norm
and higher-order Sobolev smoothness, e.g., $\psi_1(u)=\frac{1}{2}\|u\|_{\mathrm{L}^2}^2$ and
$\psi_2(u)=\frac{1}{2}\|u\|_{\mathrm{H}^1}^2$. We shall utilize a weighted (semi-)norm
$\|\cdot\|_t$ defined by
\begin{equation*}
 \|u\|_t^2=t\|L_1u\|^2+(1-t)\|L_2u\|^2,
\end{equation*}
where the weight $t\equiv t(\bs\eta)\in[0,1]$ is defined as before, and by
$Q_t=tL_1^\ast L_1+(1-t)L_2^\ast L_2$ and $L_t= Q_t^\frac{1}{2}$ and
$\widetilde{K}_t=KL_t^{-1}$. Clearly, $\|u\|_t=\|L_tu\|$. We note that the adjoint
$K^\ast$ (and hence $\widetilde{K}_t^\ast$) depends on the value $t$.
\begin{thm}
Let $\mu\in (0,1]$ be fixed, and the exact solution $u^\dagger$ satisfy the source
condition: for any $t\in[0,1]$, there exists a $w_t\in Y$ such that
$L_tu^\dagger=(\widetilde{K}_t^\ast\widetilde{K}_t)^\mu w_t$. Then for any parameter
$\bs\eta^\ast$ selected by \eqref{eqn:pPhi} with $t^\ast=t(\bs\eta^*)=\tfrac{\eta_1^\ast(\delta)}{\eta_1^\ast 
(\delta) +\eta_2^\ast (\delta)}$, the following estimate holds
\begin{equation*}
   \|u_{\bs\eta^\ast}^\delta-u^\dagger\|_{t^\ast}\leq
   C\left(\|w_{t^\ast}\|^\frac{1}{2\mu+1}+\frac{F^\frac{2+\gamma}{4}(\delta^\frac{2}{2\mu+1}\bs
   e)}{F^\frac{2+\gamma}{4}(\bs\eta^\ast)}\right)\max\{\delta_\ast,\delta\}^\frac{2\mu}{2\mu+1}.
\end{equation*}
\end{thm}
\begin{proof}
We decompose the error $u_{\bs\eta}^\delta-u^\dagger$ into
$u_{\bs\eta}^\delta-u^\dagger=(u_{\bs\eta}^\delta-u_{\bs\eta})+(u_{\bs\eta}-u^\dagger)$,
and bound the two terms separately. First we estimate the error
$u_{\bs\eta}^\delta-u_{\bs\eta}$. It follows from the optimality conditions for
$u_{\bs\eta}$ and $u_{\bs\eta}^\delta$ that
\begin{equation*}
  (K^\ast K+\eta_1L^\ast_1L_1+\eta_2L_2^\ast L_2)(u_{\bs\eta}-u_{\bs\eta}^\delta)=K^\ast(g^\dagger-g^\delta).
\end{equation*}
Multiplying the identity with $u_{\bs\eta}-u_{\bs\eta}^\delta$ and using the
Cauchy-Schwarz and Young's inequalities give
\begin{equation*}
  \begin{aligned}
     &\|K(u_{\bs\eta}^\delta-u_{\bs\eta})\|^2+\eta_1\|L_1(u_{\bs\eta}^\delta-u_{\bs\eta})\|^2
      +\eta_2\|L_2(u_{\bs\eta}^\delta-u_{\bs\eta})\|^2\\
     =&\langle K(u_{\bs\eta}^\delta-u_{\bs\eta}),g^\dagger-g^\delta\rangle\\
     \leq&\|K(u_{\bs\eta}^\delta-u_{\bs\eta})\|^2+\tfrac{1}{4}\|g^\dagger-g^\delta\|^2.
  \end{aligned}
\end{equation*}
Next let $s=\eta_1+\eta_2$. Then we get
\begin{equation*}
   \|u_{\bs\eta}^\delta-u_{\bs\eta}\|_{t}\leq\frac{\|g^\delta-g^\dagger\|}{2\sqrt{s}}\leq
   \frac{\delta}{2\sqrt{s}}\leq\frac{\delta}{2\sqrt{\max_i\eta_i}}.
\end{equation*}
Meanwhile, the minimizing property of $\bs\eta^\ast$ to the rule $\Phi_\gamma$ implies
that for any $\widehat{\bs\eta}$
\begin{equation*}
   \frac{F^{2+\gamma}(\bs\eta^\ast)}{\max(\eta_i^\ast)^2}\leq\frac{F^{2+\gamma}(\bs\eta^\ast)}{\eta_1^\ast\eta_2^\ast}\leq
   \frac{F^{2+\gamma}(\widehat{\bs\eta})}{\widehat{\eta}_1\widehat{\eta}_2}.
\end{equation*}
In particular, we may take $\widehat{\bs\eta}=\delta^\frac{2}{2\mu+1}\bs{e}$ and arrive at
\begin{equation*}
\|u_{\bs\eta^\ast}^\delta-u_{\bs\eta}\|_{t^\ast}\leq\frac{F^{\frac{2+\gamma}{4}}(\delta^\frac{2}{2\mu+1}\bs
e)}{F^{\frac{2+\gamma}{4}}(\bs\eta^\ast)}\delta^\frac{2\mu}{2\mu+1}.
\end{equation*}
Next we estimate the approximation error $u_{\bs\eta}-u^\dagger$. To this end, we observe
\begin{equation*}
\begin{aligned}
u_{\bs\eta}-u^\dagger&=(K^\ast K+\eta_1L_1^\ast L_1+\eta_2L_2^\ast
L_2)^{-1}(\eta_1L_1^\ast L_1+\eta_2L_2^\ast L_2)u^\dagger\\
&=s(K^\ast{}K+s Q_t)^{-1}Q_tu^\dagger=sL_t^{-1}(L_t^{-1}K^\ast KL_t^{-1}+sI)^{-1}L_tu^\dagger.
\end{aligned}
\end{equation*}
Hence,
$
  L_t(u_{\bs\eta}-u^\dagger)=s(\widetilde{K}_t^\ast\widetilde{K}_t+sI)L_tu^\dagger.
$
Consequently, we deduce from the source condition and the moment inequality
\cite{EnglHankeNeubauer:1996}
\begin{equation*}
\begin{aligned}
\|u_{\bs\eta}-u^\dagger\|_t&=\|L_t(u_{\bs\eta}-u^\dagger)\|
=\|s(\widetilde{K}_t^\ast\widetilde{K}_t+s I)^{-1}L_tu^\dagger\|\\
&=\|s(\widetilde{K}_t^\ast\widetilde{K}_t+s I)^{-1}(\widetilde{K}_t^\ast\widetilde{K}_t)^{\mu}w_t\|\\
&\leq\|s(\widetilde{K}_t^\ast\widetilde{K}_t+sI)^{-1}
(\widetilde{K}^\ast_t\widetilde{K}_t)^{\frac{1}{2}+\mu}w_t\|^{\frac{2\mu}{2\mu+1}}\|
s(\widetilde{K}_t^\ast\widetilde{K}_t+sI)^{-1}w_t\|^\frac{1}{2\mu+1}\\
&=\|s(\widetilde{K}_t^\ast\widetilde{K}_t+sI)^{-1}\widetilde{K}_tL_tu^\dagger\|^\frac{2\mu}{2\mu+1}\|s(\widetilde{K}_t^\ast\widetilde{K}_t+sI)^{-1}w_t\|\\
&\leq{}c(\|s(\widetilde{K}_t\widetilde{K}_t^\ast+sI)^{-1}g^\delta\|+\|s(\widetilde{K}_t\widetilde{K}_t^\ast+sI)^{-1}
(g^\delta-g^\dagger)\|)^\frac{2\mu}{2\mu+1}\|w_t\|^\frac{1}{2\mu+1},
\end{aligned}
\end{equation*}
where the constant $c$ depends only on the maximum of $r_s(t)=\frac{s}{s+t}$ over
$[0,\|\widetilde{K}_t\|^2]$. Further, we note the relation
\begin{equation*}
\begin{aligned}
s(\widetilde{K}_t\widetilde{K}_t^\ast+sI)^{-1}g^\delta
&=g^\delta-(\widetilde{K}_tK_t^\ast+sI)^{-1}\widetilde{K}_t\widetilde{K}_t^\ast g^\delta\\
&=g^\delta-\widetilde{K}(\widetilde{K}_t^\ast\widetilde{K}_t+sI)^{-1}\widetilde{K}_t^\ast g^\delta\\
&=g^\delta-K(K^\ast K+sQ_t)^{-1}K^\ast g^\delta=g^\delta-Ku_{\bs\eta}^\delta.
\end{aligned}
\end{equation*}
Hence, we deduce
\begin{equation*}
\|u_{\bs\eta^\ast}-u^\dagger\|_{t^\ast}\leq
c(\delta_\ast+c\delta)^\frac{2\mu}{2\mu+1}\|w_t\|^\frac{1}{2\mu+1}\leq
c_1\max\{\delta_\ast,\delta\}^\frac{2\mu}{2\mu+1}.
\end{equation*}
By combining these two estimates, we arrive at the desired inequality.
\end{proof}

\subsubsection{Balancing principle in Banach space}
Lastly, we turn to the balancing principle for general convex regularization
$\bs\psi$. We first recall the following technical lemma \cite{JinLorenz:2010} for single
convex regularization $\psi$. The first estimates the propagation error, and the second
plays the role of a triangle inequality.

\begin{lem}[{\cite{JinLorenz:2010}}]\label{lem:resest}
Let the exact solution $u^\dagger$ satisfy the following source condition: there exists a
$w\in Y$ such that $K^\ast w=\xi\in\partial\psi(u^\dagger)$, and let
$\xi_\eta=K^\ast(g^\dagger-Ku_\eta)/\eta$. Then there hold
\begin{equation*}
  \begin{aligned}
     &d_{\xi_\eta}(u_\eta^\delta,u_\eta)\leq \frac{\delta^2}{2\eta}\quad \mbox{and}\quad \|K(u_\eta^\delta-u_\eta)\|\leq 2\delta,\\
     &\left|d_\xi(u_\eta^\delta,u^\dagger)-(d_{\xi_\eta}(u_\eta^\delta,u_\eta)+d_\xi(u_\eta,u^\dagger))\right|\leq 6\|w\|\delta.
  \end{aligned}
\end{equation*}
\end{lem}

Now we can state an estimate for the balancing principle \eqref{eqn:pPhi} in Banach spaces.
The estimate has been stated in \cite{ItoJinTakeuchi:2011} but without a proof.
\begin{thm}\label{thm:errest}
Let the exact solution $u^\dagger$ satisfy the source condition: for any $t\in[0,1]$
there exists a $w_t\in Y$ such that $K^\ast
w_t=\xi_t\in\partial\left([t,1-t]^\mathrm{t}\cdot\bs\psi(u^\dagger)\right)$. Then for
every $\bs\eta^\ast$ selected by \eqref{eqn:pPhi} and with with $t^\ast=t(\bs\eta^*) =\tfrac{\eta_1^\ast(\delta)}{\eta_1^\ast (\delta) +\eta_2^\ast (\delta)}$, the following estimate holds
\begin{equation*}
  d_{\xi_{t^\ast}}(u_{\bs\eta^\ast}^\delta,u^\dagger)\leq C\left(\|w_{t^\ast}\|+\frac{F^{1+\frac{\gamma}{2}}
  (\delta\bs e)}{F^{1+\frac{\gamma}{2}}(\bs\eta^\ast)}\right)\max(\delta,\delta_\ast).
\end{equation*}
\end{thm}
\begin{proof}
For any $t\in[0,1]$, let $\psi_{t}(u^\dagger) =[t,1-t]^\mathrm{t}\cdot\bs\psi(u^\dagger)$
and $\xi_t\in\partial \psi_t(u^\dagger)$, with $\xi_t$ and $w_t$ being the subgradient
and the representer in the source condition, respectively. By Lemma \ref{lem:resest}, we
have that for $\bs{\eta}$
\begin{equation*}
d_{\xi_t}(u_{\bs\eta}^\delta,u^\dagger)\leq d_{\xi_{\bs\eta}}(u_{\bs\eta}^\delta,u_{\bs\eta}) +
d_{\xi_t}(u_{\bs\eta},u^\dagger)+ 6\|w_t\|\delta,
\end{equation*}
where $\xi_{\bs\eta}=-K^\ast(K(u_{\bs\eta})-g^\dagger)/s\in\partial\psi_t(u_{\bs\eta})$
and $s=\eta_1+\eta_2$. It suffices to bound the terms involving Bregman distance. We
first estimate the approximation error $d_{\xi_t}(u_{\bs\eta},u^\dagger)$. To this end,
observe by the minimizing property of the element $u_{\bs\eta}$, i.e.,
\begin{equation*}
  \tfrac{1}{2}\|Ku_{\bs\eta}-g^\dagger\|^2+s\psi_t(u_{\bs\eta})\leq\tfrac{1}{2}
  \|Ku^\dagger-g^\dagger\|^2+s\psi_t(u^\dagger)=s\psi_t(u^\dagger).
\end{equation*}
This inequality, the definition of $d_{\xi_t}(u_{\bs\eta},u^\dagger)$, the source
condition and Lemma \ref{lem:resest} implies
\begin{equation*}
\begin{aligned}
d_{\xi_t}(u_{\bs\eta},u^\dagger)&\leq -\langle w_t, K(u_{\bs\eta}-u^\dagger)\rangle\\
&\leq \|w_t\|\|Ku_{\bs\eta}-g^\dagger\|\\
&\leq \|w_t\|\left(\|K(u_{\bs\eta}-u_{\bs\eta}^\delta)\|+
\|Ku_{\bs\eta}^\delta-g^\delta\|+\|g^\delta-g^\dagger\|\right)\\
&\leq\|w_t\|(2\delta+\delta_\ast+\delta)\leq
4\|w_t\|\max(\delta,\delta_\ast).
\end{aligned}
\end{equation*}

Next we estimate  the term $d_{\xi_{\bs\eta}}(u_{\bs\eta}^\delta, u_{\bs\eta})$. In view
of Lemma \ref{lem:resest}, we have
\begin{equation*}
d_{\xi_{\bs\eta}}(u_{\bs\eta}^\delta,u_{\bs\eta})\leq \frac{\delta^2}{2s}\leq \frac{\delta^2}{2\max(\eta_i)}.
\end{equation*}
Meanwhile, the minimizing property of $\bs\eta$ to the rule $\Phi_\gamma$ gives that for
any $\widehat{\bs\eta}$
\begin{equation*}
\frac{F^{2+\gamma}(\bs\eta)}{\max(\eta_i)^2}\leq\frac{F^{2+\gamma}(\bs\eta)}{\eta_1\eta_2}\leq
\frac{F^{2+\gamma}(\widehat{\bs\eta})}{\widehat{\eta}_1\widehat{\eta}_2}.
\end{equation*}
Upon letting $\widehat{\bs\eta}=\delta\bs e$ and combining the preceding two
inequalities, we get
\begin{equation*}
d_{\xi_{\bs\eta}}(u_{\bs\eta}^\delta,u_{\bs\eta})\leq
\frac{F(\delta\bs e)^{1+\frac{\gamma}{2}}}{F(\bs\eta)^{1+\frac{\gamma}{2}}}\frac{\delta}{2}.
\end{equation*}
Now combining these three estimates gives the desired assertion.
\end{proof}

The a posteriori error estimate in Theorem 2.4 coincides with that for the a priori
choice, e.g., $\bs\eta\sim\delta\bs e$, provided that the realized discrepancy
$\delta_\ast$ is of the same order with the exact noise level $\delta$.

\section{Numerical algorithms}\label{sec:alg}
Now we describe algorithms for numerically realizing the hybrid
principle and the balancing principle, i.e.,
Broyden's method and fixed-point algorithm, and discuss their properties.

\subsection{Broyden's method}
In practice, the application of the hybrid principle invokes solving the nonlinear system
\eqref{aT}, which is nontrivial due to its potential nonsmoothness and high degree of
nonlinearity. We propose using Broyden's method
\cite{Broyden:1965} for its efficient solution; see Algorithm
\ref{alg:broyd} for a complete description.

For the numerical treatment, we reformulate system \eqref{aT} equivalently as
\begin{equation*}
  \mathbf{T}(\bs\eta)\equiv\left(\begin{aligned}
   \phi(u_{\bs\eta}^\delta,g^\delta) - \tfrac{1}{2}\delta^2 + \eta_2\psi_2(u_{\bs\eta}^\delta) - \eta_1\psi_1(u_{\bs\eta}^\delta)\\
   \phi(u_{\bs\eta}^\delta,g^\delta) - \tfrac{1}{2}\delta^2 + \eta_1\psi_1(u_{\bs\eta}^\delta) - \eta_2\psi_2(u_{\bs\eta}^\delta)
  \end{aligned}\right)=0.
\end{equation*}
The system is numerically more amenable than \eqref{eqn:rbsys}.
In Algorithm \ref{alg:broyd}, the Jacobian $\mathbf{J}_0$ can be approximated by finite
difference. Step 7 represents the celebrated Broyden update. The stopping criterion
is based on monitoring the residual norm $\|\mathbf{T}(\bs\eta)\|$. Note that each iteration
involves evaluating $\mathbf{T}(\bs\eta)$, which in turn incurs solving one optimization problem
of minimizing $J_{\bs\eta}$. Our experiences
indicate that it converges fast and steadily, however, a convergence analysis
is still missing.
\begin{algorithm}[ht!]
  \caption{Broyden's method for system \eqref{aT}.}\label{alg:broyd}
  \begin{algorithmic}[1]
    \STATE Set $k=0$ and choose $\bs\eta^0$.
    \STATE Compute the Jacobian $\mathbf{J}_0=\nabla \mathbf{T}(\bs\eta^0)$ and equation residual $\mathbf{T}(\bs\eta^0)$.
    \FOR {$k=1,\ldots,K$}
    \STATE Calculate the quasi-Newton update $\Delta\bs\eta = -\mathbf{J}_{k-1}^{-1}\mathbf{T}(\bs\eta^{k-1})$.
    \STATE Update the regularization parameter $\bs\eta$ by $\bs\eta^{k}=\bs\eta^{k-1}+\Delta\bs\eta$.
    \STATE Evaluate the equation residual $\mathbf{T}(\bs\eta^{k})$ and set $\Delta\mathbf{T}=\mathbf{T}(\bs\eta^k)-
           \mathbf{T}(\bs\eta^{k-1})$.
    \STATE Compute Jacobian update
           \begin{equation*}
               \mathbf{J}_k=\mathbf{J}_{k-1}+\frac{1}{\|\Delta\bs\eta\|^2}[\Delta\mathbf{T}-\mathbf{J}_k\Delta\bs\eta]
               \cdot\Delta\bs\eta^\mathrm{t}.
           \end{equation*}
    \STATE Check the stopping criterion.
    \ENDFOR
    \STATE Output the solution
  \end{algorithmic}
\end{algorithm}

\subsection{Fixed point algorithm}
In this part, we describe a fixed point algorithm for computing the minimizer of the rule
$\Phi_\gamma$. The algorithm was originally introduced in \cite{ItoJinTakeuchi:2011}, but without
any analysis. One basic version is listed in Algorithm \ref{alg:fpa1}, where the
subscript $-i$ refers to the index different from $i$. The stopping
criterion at Step 4 can be based on monitoring the relative change of the regularization
parameter $\bs\eta$ or the inverse solution $u_{\bs\eta}^\delta$.

\begin{algorithm}[ht!]
  \caption{Fixed point algorithm for minimizing \eqref{eqn:pPhi}.}\label{alg:fpa1}
  \begin{algorithmic}[1]
    \STATE Set $k=0$ and choose $\bs\eta^0$.
    \STATE Solve for $u^{k+1}$ by the Tikhonov regularization
      \begin{equation*}
          u^{k+1}=\arg\min_{u}\left\{\phi(u,g^\delta)+\bs\eta^k\cdot\bs\psi(u)\right\}.
      \end{equation*}
    \STATE Update the regularization parameter $\bs\eta^{k+1}$ by
      \begin{equation*}
         \begin{aligned}
            \eta_i^{k+1}&=\frac{1}{1+\gamma}\frac{\phi(u^{k+1},g^\delta)+\eta_{-i}^k\psi_{-i}(u^{k+1})}{\psi_i(u^{k+1})},\quad i=1,2.
         \end{aligned}
      \end{equation*}
    \STATE Check the stopping criterion.
  \end{algorithmic}
\end{algorithm}

We shall analyze Algorithm \ref{alg:fpa1}. First, we introduce a fixed point operator
$\mathbf{T}$ by
\begin{equation*}
  \mathbf{T}(\bs\eta)=(1+\gamma)^{-1}\left(\begin{aligned}
      \frac{\phi(u_{\bs\eta}^\delta,g^\delta)+\eta_2\psi_2(u_{\bs\eta}^\delta)}{\psi_1(u_{\bs\eta}^\delta)}\\
      \frac{\phi(u_{\bs\eta}^\delta,g^\delta)+\eta_1\psi_1(u_{\bs\eta}^\delta)}{\psi_2(u_{\bs\eta}^\delta)}
      \end{aligned}\right).
\end{equation*}

We shall also need the next result \cite[Lem. 2.1 and Cor. 2.3]{Ito:2010a}.
\begin{lem}\label{lem:mon}
The function $\psi_i(u_{\bs\eta}^\delta)$ is monotonically decreasing in $\eta_i$, and
the following relations hold
\begin{equation*}
 \frac{\partial}{\partial\eta_i}(\phi(u_{\bs\eta}^\delta,g^\delta)+\eta_{-i}\psi_{-i}(u_{\bs\eta}^\delta))
 + \eta_i\frac{\partial}{\partial \eta_i}\psi_i(u_{\bs\eta}^\delta)=0,\quad i=1,2.
\end{equation*}
\end{lem}

We have the next monotone result for the fixed point operator $\mathbf{T}$.
\begin{prop}\label{prop:monoalgi}
Let the function $F(\bs\eta)$ be twice differentiable. Then the map $\mathbf{T}(\bs\eta)$
is monotone if $F^2(\bs\eta)F_{\eta_1\eta_1}(\bs\eta)F_{\eta_2\eta_2} (\bs\eta)>
(F_{\eta_1}(\bs\eta)F_{\eta_2}(\bs\eta)-F(\bs\eta)F_{\eta_1\eta_2} (\bs\eta))^2$.
\end{prop}
\begin{proof}
Let $A(\bs\eta)=\phi+\eta_2\psi_2$ and $B(\bs\eta)=\phi+\eta_1\psi_1$. By Lemma
\ref{lem:mon}, there hold
\begin{equation}\label{eqn:varderiv}
\frac{\partial A}{\partial \eta_1}+\eta_1\frac{\partial\psi_1
}{\partial \eta_1}=0\quad\mbox{and}\quad \frac{\partial B}{\partial
\eta_2}+\eta_2\frac{\partial\psi_2 }{\partial \eta_2}=0.
\end{equation}
With the help of these two relations, we deduce
\begin{equation*}
  \begin{aligned}
    \frac{\partial}{\partial \eta_1}\frac{A}{\psi_1}&=
      \frac{\psi_1\frac{\partial A}{\partial
      \eta_1}-A\frac{\partial\psi_1}{\partial\eta_1}}{\psi_1^2}=\frac{\psi_1(
      -\eta_1\frac{\partial\psi_1}{\partial\eta_1})-A\frac{\partial\psi_1}{\partial\eta_1}}{\psi_1^2}\\
      &=-\frac{1}{\psi_1^2}\frac{\partial\psi_1}{\partial\eta_1}(\eta_1\psi_1+A)=-\frac{F}{\psi_1^2}\frac{\partial\psi_1}{\partial\eta_1},
  \end{aligned}
\end{equation*}
and
\begin{equation*}
  \begin{aligned}
    \frac{\partial}{\partial\eta_2}\frac{A}{\psi_1}&=\frac{\psi_1\frac{\partial
    A}{\partial\eta_2}-A\frac{\partial\psi_1}{\partial\eta_2}}{\psi_1^2}
    =\frac{\psi_1\frac{\partial}{\partial\eta_2}(F-\eta_1\psi_1)-(F-\eta_1\psi_1)\frac{\partial\psi_1}{\partial\eta_2}}{\psi_1^2}\\
    &=\frac{1}{\psi_1^2}\left[\psi_1\psi_2-F\frac{\partial\psi_1}{\partial\eta_2}\right],
  \end{aligned}
\end{equation*}
where we have used the relation $\frac{\partial F}{\partial\eta_2}=\psi_2$ from Lemma
\ref{lem:diff}. Similarly, we have
\begin{equation*}
   \frac{\partial}{\partial\eta_2}\frac{B}{\psi_2}=-\frac{F}{\psi_2^2}\frac{\partial\psi_2}{\partial\eta_2}\quad\mbox{and}\quad
   \frac{\partial}{\partial\eta_1}\frac{B}{\psi_2}=\frac{1}{\psi_2^2}\left[\psi_1\psi_2-F\frac{\partial\psi_2}{\partial\eta_1}\right].
\end{equation*}
Therefore, the Jacobian $\nabla\mathbf{T}$ of the operator $\mathbf{T}$ is given by
\begin{equation*}\displaystyle
   \nabla\mathbf{T}=(1+\gamma)^{-1}\left(\begin{array}{cc}
       -\frac{F}{\psi_1^2}\frac{\partial\psi_1}{\partial\eta_1}
       &\frac{1}{\psi_1^2}\left[\psi_1\psi_2-F\frac{\partial\psi_1}{\partial\eta_2}\right]\\
       \frac{1}{\psi_2^2}\left[\psi_1\psi_2-F\frac{\partial\psi_2}{\partial\eta_1}\right]
       &-\frac{F}{\psi_2^2}\frac{\partial\psi_2}{\partial\eta_2}
   \end{array}\right).
\end{equation*}

Now Lemma \ref{lem:mon} implies that $-\frac{\partial\psi_i}{\partial \eta_i} \geq 0$.
Hence, it suffices to show that the determinant $|\nabla \mathbf{T}|>0$. By Lemma \ref{lem:diff}, the identity
$\frac{\partial \psi_1}{\partial\eta_2}=F_{\eta_1\eta_2}=\frac{\partial
\psi_2}{\partial\eta_1}$ holds, and thus $|\nabla \mathbf{T}|$ is given by
\begin{equation*}
|\nabla\mathbf{T}|=(1+\gamma)^{-1}\frac{1}{\psi_1^2\psi_2^2}\left[F^2\frac{\partial \psi_1}{\partial \eta_1}
\frac{\partial \psi_2}{\partial \eta_2}-\left(\psi_1\psi_2-F\frac{\partial\psi_2}{\partial\eta_1}\right)^2\right].
\end{equation*}
Hence, the nonnegativity of $|\nabla\mathbf{T}|$ follows from the assumption
$F^2(\bs\eta)F_{\eta_1\eta_1}(\bs\eta) F_{\eta_2\eta_2}(\bs\eta)-(F_{\eta_1}(\bs\eta)
F_{\eta_2}(\bs\eta)-F(\bs\eta)F_{\eta_1\eta_2}(\bs\eta))^2>0$. This concludes the proof.
\end{proof}

\section{Numerical experiments}\label{sec:exp}

We now provide some numerical results for
the hybrid principle \eqref{aT}; and the balancing principle
\eqref{eqn:pPhi} has been numerically exemplified in \cite{ItoJinTakeuchi:2011} and will
not be addressed here. The examples are integral equations of the first kind with kernel $k(s,t)$ and
solution $u(t)$. All the examples are taken from \cite{ItoJinTakeuchi:2011}.
The discretized linear system takes the form $\mathbf{Ku}^\dagger=\mathbf{g}^\dagger$.
The data $\mathbf{g}^\dagger$ is then corrupted by noises, i.e., $g_i^\delta=g_i^\dagger+
\max_i\{|g_i^\dagger|\}\varepsilon\zeta_i$, where $\zeta_i$ are standard Gaussian variables, 
and $\varepsilon$ is the relative noise level.

\subsection{$\mathrm{H}^1$-$\mathrm{TV}$ model}

\begin{exam}\label{exam:h1tv}
Let $\xi(t)=\chi_{|t|\leq3}(1+\cos\frac{\pi t}{3}$, and the kernel $k(s,t)$ is given by 
$\xi(s-t)$. The true solution $u^\dagger$ exhibits both flat and smoothly varying regions 
and it is shown in Fig. \ref{fig:h1tv}, and the integration interval is $[-6,6]$. We adopt 
two penalties $\psi_1(u)=|u|^2_{\mathrm{H}^1}$ and $\psi_2(u)=|u|_\mathrm{TV}$.
\end{exam}

\begin{table}[htb!]
\centering \caption{Numerical results for Example \ref{exam:h1tv}.}
\begin{tabular}{ccccccccc}
\hline
$\epsilon$&$\bs\eta_\mathrm{bdp}$&$\bs\eta_\mathrm{opt}$&$\eta_\mathrm{h1}$&$\eta_\mathrm{tv}$&$e_\mathrm{bdp}$&$e_\mathrm{opt}$&$e_\mathrm{h1}$&$e_\mathrm{tv}$\\
\hline
5e-2&(5.89e-3,9.67e-3)&(2.30e-4,2.05e-3)&6.17e-4&9.67e-3&3.50e-2&2.65e-2&3.96e-2&1.07e-1\\
5e-3&(3.41e-4,5.98e-4)&(2.34e-5,3.92e-4)&8.34e-5&4.51e-4&2.45e-2&1.09e-2&2.70e-2&9.49e-2\\
5e-4&(2.93e-6,5.41e-6)&(2.55e-6,4.48e-5)&1.26e-6&5.16e-5&1.22e-2&8.86e-3&1.38e-2&4.49e-2\\
5e-5&(1.19e-7,2.26e-7)&(5.88e-8,4.36e-6)&8.98e-8&3.79e-6&6.91e-3&5.53e-3&9.40e-3&1.68e-2\\
5e-6&(4.94e-9,9.50e-9)&(1.93e-10,6.22e-9)&5.18e-10&2.80e-7&4.64e-3&2.90e-3&5.29e-3&5.13e-3\\
\hline
\end{tabular}\label{tab:h1tv}
\end{table}

\begin{figure}
\centering
\begin{tabular}{ccc}
\includegraphics[trim = 1cm 0cm 2cm .5cm,width=5.2cm]{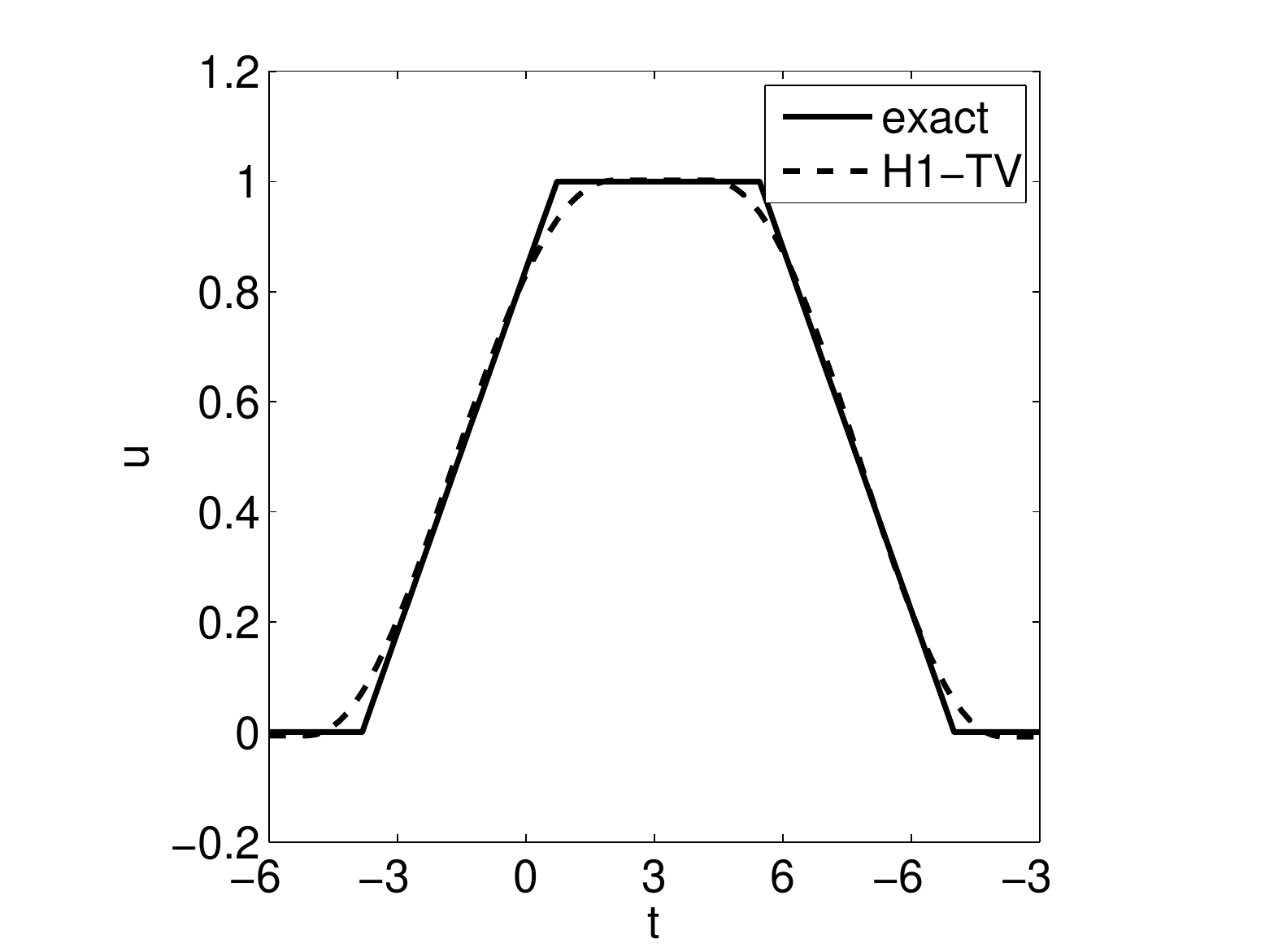}&
\includegraphics[trim = 1cm 0cm 2cm .5cm,width=5.2cm]{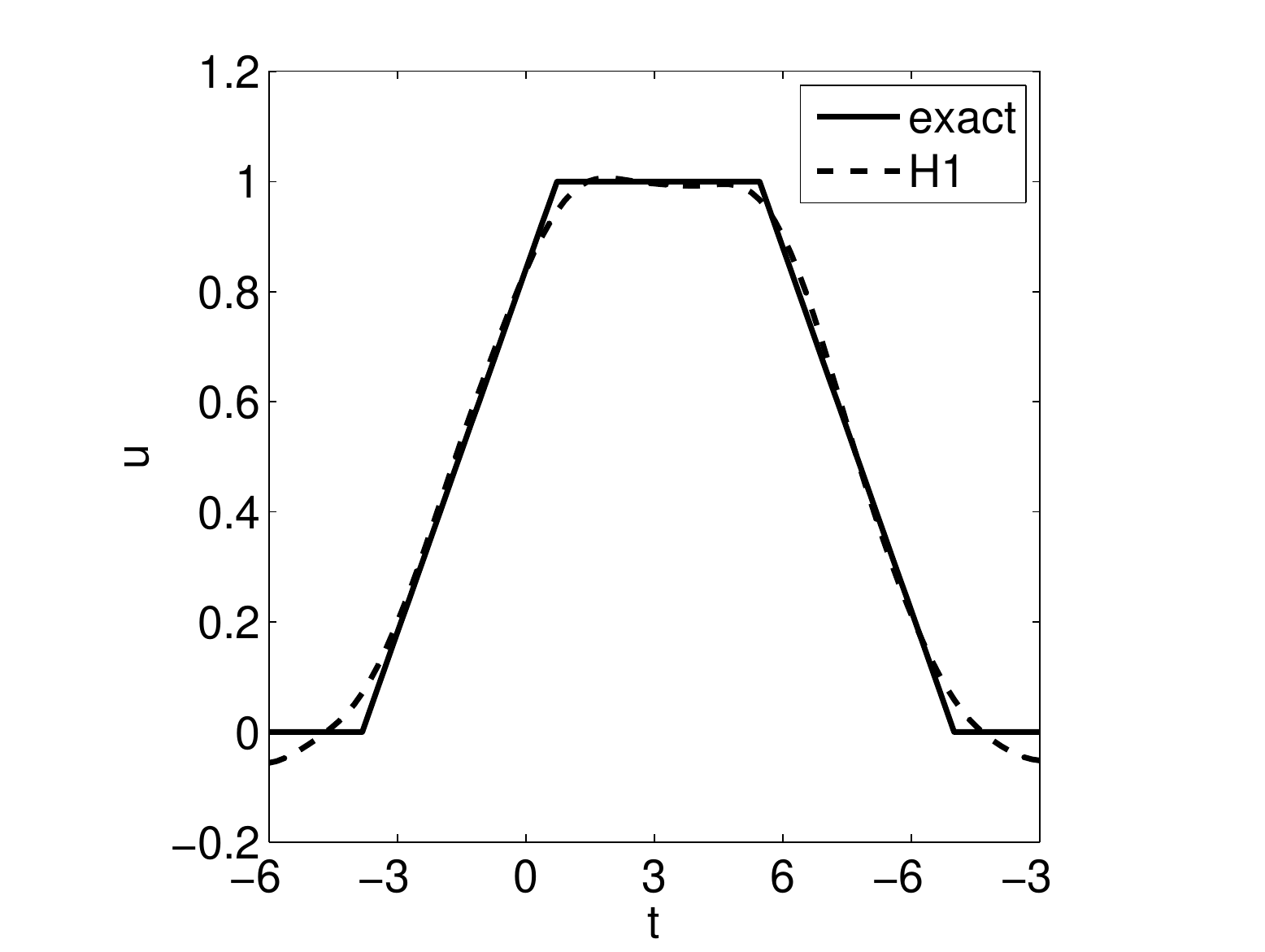}&
\includegraphics[trim = 1cm 0cm 2cm .5cm,width=5.2cm]{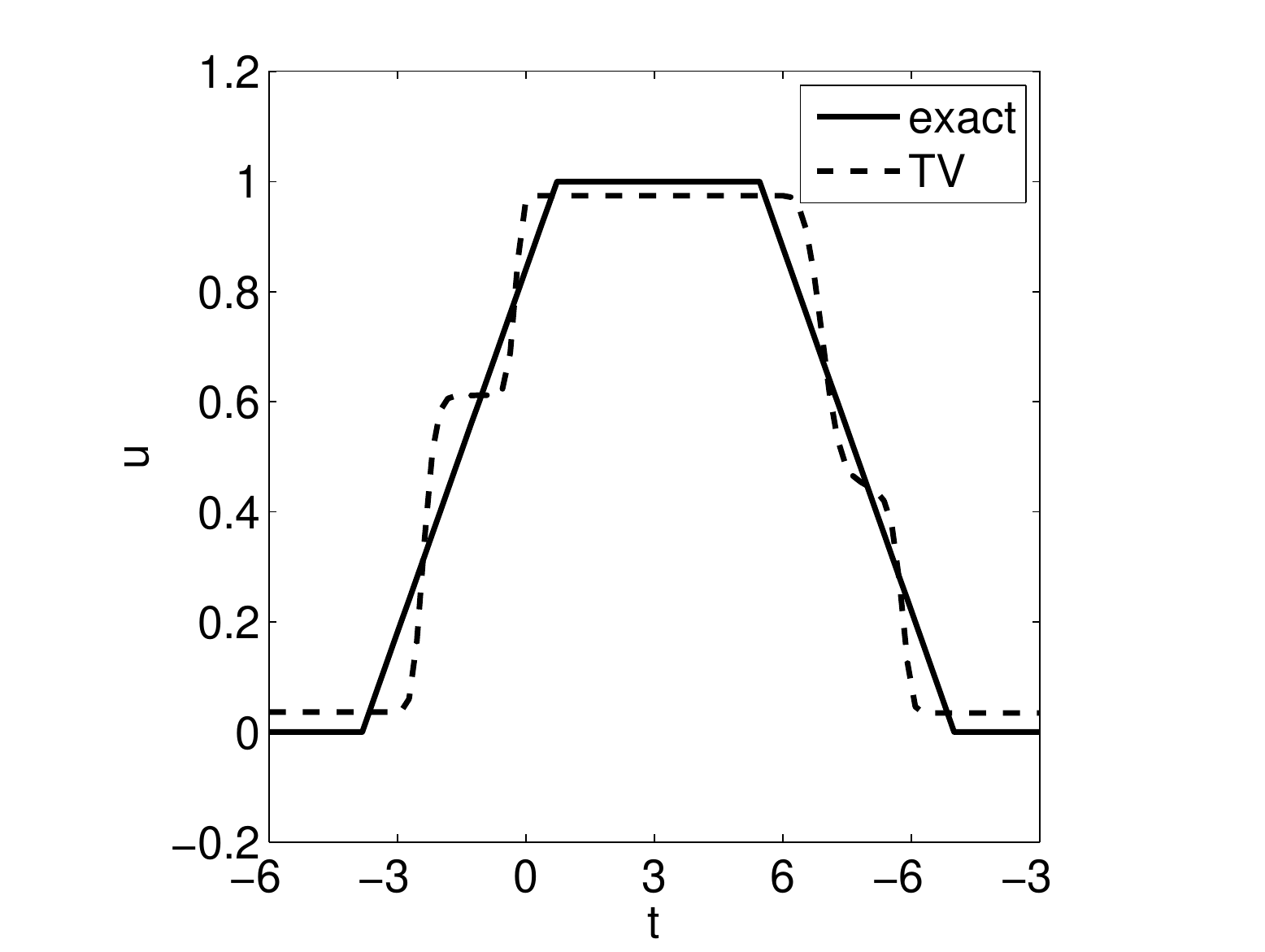}\\
$u_\mathrm{bdp}$ & $u_\mathrm{h1}$ & $u_\mathrm{tv}$
\end{tabular}
\caption{Numerical results for Example \ref{exam:h1tv} with $\varepsilon=5\%$ noise.}\label{fig:h1tv}
\end{figure}

The numerical results are summarized in Table \ref{tab:h1tv}. In the table, the
subscripts $\mathrm{bdp}$ and $\mathrm{opt}$ respectively refer to the hybrid principle and the optimal
choice, i.e., the value giving the smallest error. The
single-parameter models are indicated by subscripts $\mathrm{h1}$ and
$\mathrm{tv}$, and the regularization parameter shown in Table \ref{tab:h1tv} is the
optimal one. The accuracy of the results is measured by the relative $L^2$ error
$e=\|u-u^\dagger\|_{\mathrm{L}^2}/\|u^\dagger\|_{\mathrm{L}^2}$. We observe that the 
$\mathrm{H}^1$-$\mathrm{TV}$ model in conjunction with the hybrid principle achieves a smaller error than
either $\mathrm{H}^1$ or $\mathrm{TV}$ with the optimal choice, thereby showing the advantages of
the $\mathrm{H}^1$-$\mathrm{TV}$ model. Further, the hybrid principle gives an error fairly close
to the optimal one, within a factor of two, and the error decreases as the noise level decreases.

Let us briefly comment on the performance of the multi-parameter model. The classical $\mathrm{H}^1$ model recovers
the flat region unsatisfactorily, whereas the $\mathrm{TV}$ approach clearly suffers from staircasing effect
in the gray region and reduced magnitude in the flat region, cf. Fig. \ref{fig:h1tv}. In
contrast, the $\mathrm{H}^1$-$\mathrm{TV}$ model preserves the magnitude of flat region while
recovering the gray region excellently. Therefore, the $\mathrm{H}^1$-$\mathrm{TV}$ model does combine the strengths of
both $\mathrm{H}^1$ and $\mathrm{TV}$ models. Finally, we would like to remark that Broyden's method
converges rapidly with the convergence achieved in five iterations, and the convergence behavior
is not sensitive to the initial guess.

\subsection{Elastic-net model}
\begin{exam}\label{exam:l1l2}
The kernel $k(s,t)$ is given by $\tfrac{1}{4}\left(\tfrac{1}{16}+ (s-t)^2\right)^{
-\frac{3}{2}}$, the exact solution $u^\dagger$ consists of two bumps and it is shown in
Fig. \ref{fig:l1l2}. The penalties are $\psi_1(u)=\|u\|_{\ell^1}$ and
$\psi_2(u)=\frac{1}{2} \|u\|_{\ell^2}^2$ to retrieve the groupwise sparsity structure, 
which is known as elastic-net in statistics \cite{ZouHastie:2005}.
The integration interval is $[0,1]$. The size of the problem is 100.
\end{exam}

\begin{figure}
\centering
\begin{tabular}{ccc}
\includegraphics[trim = 1cm 0cm 2cm .5cm, clip=true,width=5.2cm]{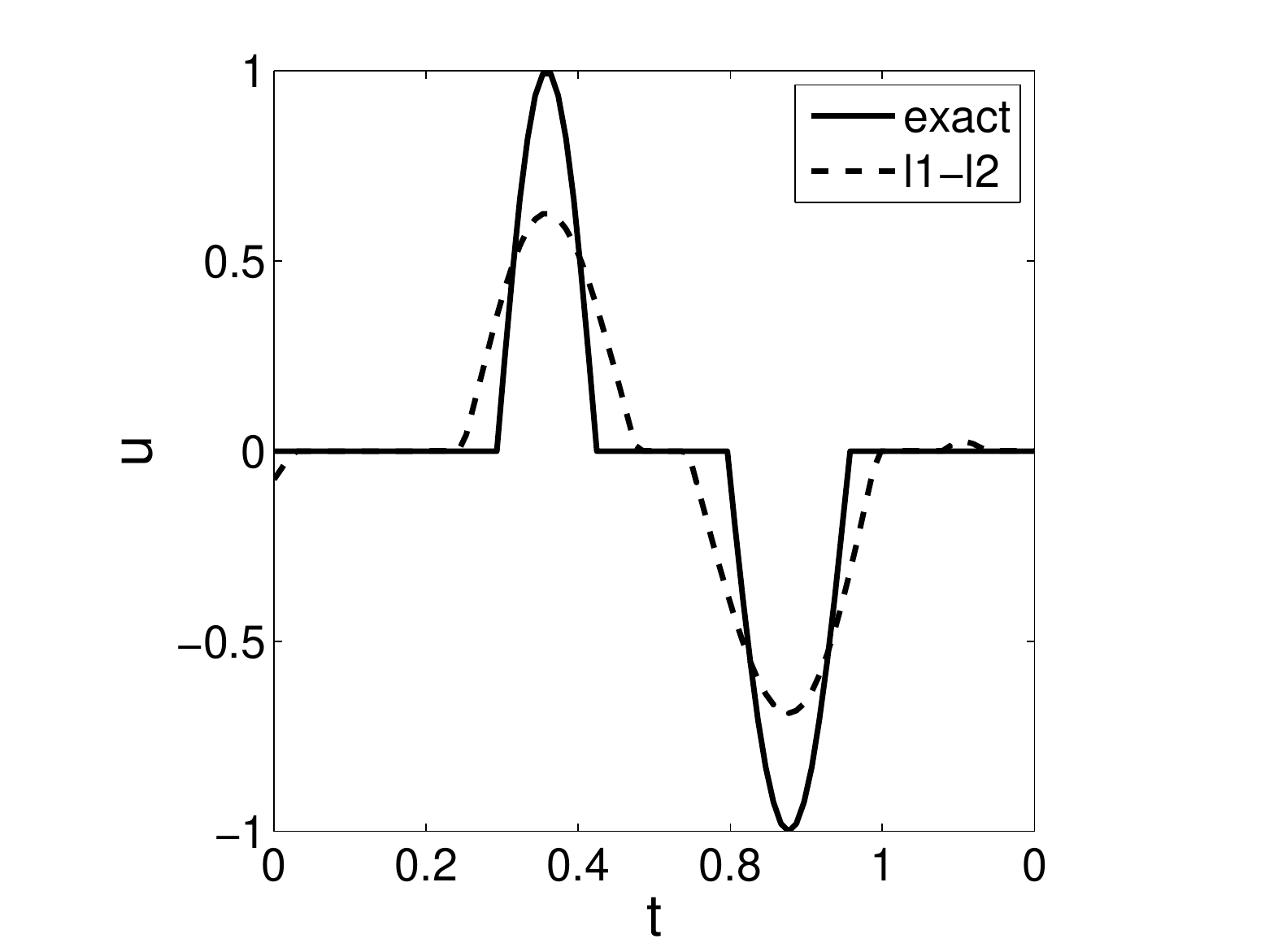}&
\includegraphics[trim = 1cm 0cm 2cm .5cm, clip=true,width=5.2cm]{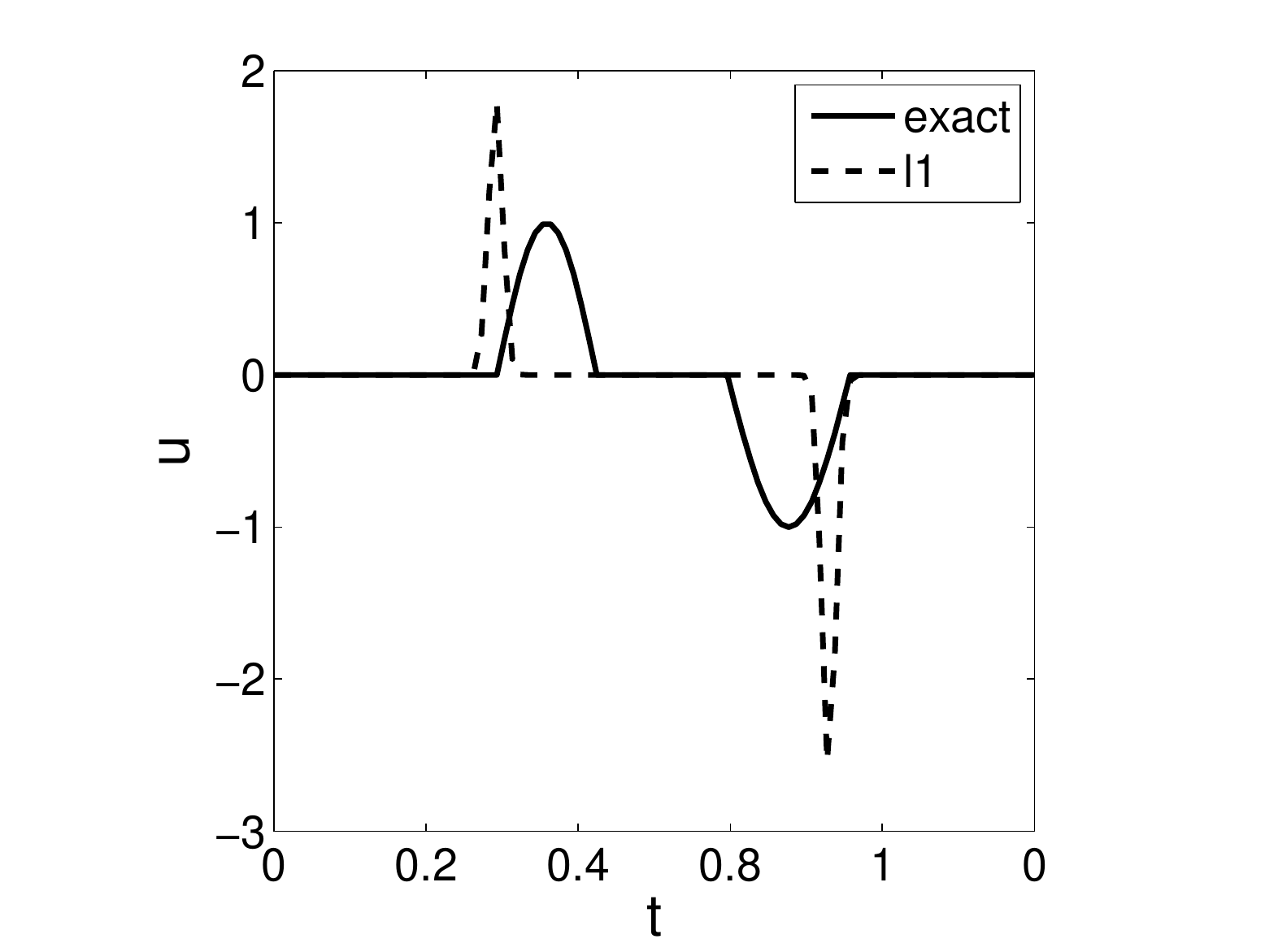}&
\includegraphics[trim = 1cm 0cm 2cm .5cm, clip=true,width=5.2cm]{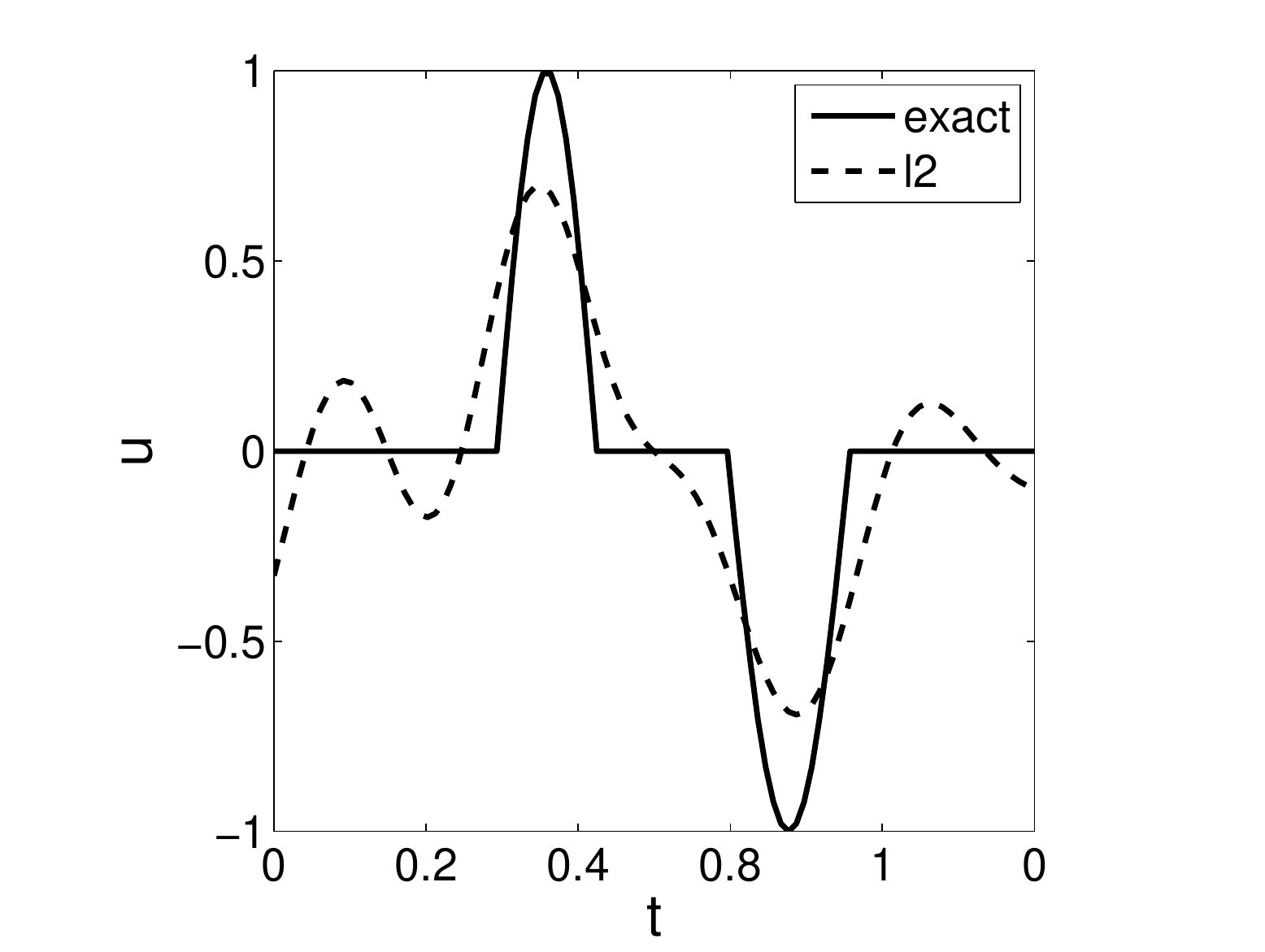}\\
$u_\mathrm{bdp}$ &$u_\mathrm{l1}$ & $u_\mathrm{l2}$
\end{tabular}
\caption{Numerical results for Example \ref{exam:l1l2} with $\varepsilon=5\%$ noise.}\label{fig:l1l2}
\end{figure}

\begin{table}[ht!]
\centering\caption{Numerical results for Example \ref{exam:l1l2}.}
\begin{tabular}{ccccccccc}
\hline
$\epsilon$&$\bs\eta_\mathrm{bdp}$&$\bs\eta_\mathrm{opt}$&$\eta_\mathrm{l1}$&$\eta_\mathrm{l2}$&$e_\mathrm{bdp}$&$e_\mathrm{opt}$&$e_\mathrm{l1}$&$e_\mathrm{l2}$\\
\hline
5e-2&(2.44e-3,9.60e-3)&(2.81e-3,1.16e-3)&1.16e0 &3.11e-3&4.09e-1&8.57e-2&1.29e0 &4.58e-1\\
5e-3&(7.30e-5,2.25e-4)&(2.59e-4,1.11e-4)&9.67e-5&3.13e-5&1.96e-1&1.20e-2&9.00e-1&2.90e-1\\
5e-4&(4.73e-6,1.27e-5)&(2.23e-5,1.11e-5)&1.27e-5&4.13e-6&7.50e-2&8.18e-3&6.18e-1&2.17e-1\\
5e-5&(3.29e-7,8.42e-7)&(2.73e-6,1.28e-6)&1.12e-6&3.79e-8&2.01e-2&4.69e-3&4.85e-1&1.66e-1\\
5e-6&(2.56e-8,6.50e-8)&(1.60e-7,9.92e-8)&5.14e-9&1.25e-9&1.16e-2&2.27e-3&2.62e-1&9.55e-2\\
\hline
\end{tabular}\label{tab:l1l2}
\end{table}

It is observed from Table \ref{tab:l1l2} that the hybrid principle gives slightly too 
small but otherwise reasonable estimate for the optimal choice. A close look at Fig. 
\ref{fig:l1l2} indicates that the solution $u_\mathrm{l2}$ has almost no zero entries, and
thus it fails to distinguish between relevant and irrelevant factors. Meanwhile,
many entries of the $\ell^1$ solution are zero, and thus some relevant factors
are correctly identified. However, it tends to select only a part instead of all relevant
factors. The elastic-net combines the best of both $\ell^1$
and $\ell^2$ models, and it achieves the desired goal of identifying the group structure.

\subsection{Image deblurring}
\begin{exam}\label{exam:l1l22d}
The kernel $k(s,t)$
performs standard Gaussian blur with standard deviation $1$ and blurring width $5$. The
exact solution $u^\dagger$ is shown in Fig. \ref{fig:l1l22d}. The size of the image is
$50\times50$.  The penalties are
$\psi_1(u)=\|u\|_{\ell^1}$ and $\psi_2(u)=\frac{1}{2} \|u\|_{\ell^2}^2$.
\end{exam}

\begin{figure}\centering
\begin{tabular}{ccc}
\includegraphics[width=3.5cm]{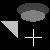}&\includegraphics[width=3.5cm]{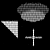}&\\
 $u^\dagger$& $u_\mathrm{opt}$ & \\
\includegraphics[width=3.5cm]{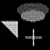}&\includegraphics[width=3.5cm]{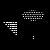}
&\includegraphics[width=3.5cm]{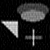}\\
$u_\mathrm{bdp}$& $u_\mathrm{l1}$ & $u_\mathrm{l2}$
\end{tabular}
\caption{Numerical results for Example \ref{exam:l1l22d} with $\varepsilon=1\%$
noise. The selected regularization parameters are
$\bs\eta_\mathrm{bdp}$=(4.70e-3,4.65e-3), $\bs\eta_\mathrm{opt}$
=(1.26e-2,1.31e-3), $\eta_\mathrm{l1}$=5.67e-1, and $\eta_\mathrm{l2}$=3.51e-3.}\label{fig:l1l22d}
\end{figure}

This example represents a more realistic problem of image deblurring. Here one half of the
data points are retained, which renders the problem far more ill-posed. The $\ell^1$ solution
is very spiky, cf. Fig. \ref{fig:l1l22d}, and neighboring pixels
act independently of each other. In particular, many pixels
in the blocks and the cross are missing. In contrast, the solution $u_\mathrm{l2}$ is smooth,
but there are many small spurious
oscillations in the background. The elastic-net model achieves the best of the two:
retaining the block structure with only few spurious nonzero coefficients. The numbers are also very telling:
$e_\mathrm{bdp}=$2.96e-1, $e_\mathrm{o}$=2.44e-1,
$e_\mathrm{l1}=$9.21e-1, and $e_\mathrm{l2}$=3.42e-1. Hence, the
error $e_\mathrm{bdp}$ agrees well with the optimal choice, and it is
smaller than that with the optimal choice for either $\ell^1$ or $\ell^2$ models.

\section{Conclusions}
We have studied multi-parameter regularization from the viewpoint of
augmented Tikhonov regularization, and shown a unified way to derive the balancing principle and
balanced discrepancy principle. A priori and a posteriori error estimates for the principles were provided,
and efficient numerical algorithms (Broyden's method and fixed point algorithm) were
presented and discussed. Numerical results were presented to illustrate the feasibility of the balanced discrepancy principle.

\section*{Acknowledgements}
This work was partially carried out during
the visit of K.I. at Institute for Applied Mathematics and
Computational Science of Texas A\&M University. He would like to thank the institute for
the hospitality.
\bibliographystyle{abbrv}
\bibliography{multitikh}
\end{document}